\documentclass[english, final,leqno,onefignum,onetabnum]{siamltex1213}
\usepackage[T1]{fontenc}
\usepackage[latin9]{inputenc}
\usepackage{color}
\usepackage{amsmath}
\usepackage{amssymb}
\usepackage{setspace}
\usepackage{graphics, epsfig}
\begin{document}

\title{Some Extensions of the Crouzeix-Palencia Result}
\author{Trevor Caldwell%
\and
 Anne Greenbaum%
\and
 Kenan Li%
    \thanks{University of Washington, Applied Mathematics Dept., Box 353925,
            Seattle, WA 98195.  This work was supported in part by NSF grant DMS-1210886.}
}

\maketitle

\begin{abstract}
In [{\em The Numerical Range is a $(1 + \sqrt{2})$-Spectral Set}, SIAM J. Matrix Anal. Appl. 38 (2017), 
pp.~649-655],
Crouzeix and Palencia show that the closure of the
numerical range of a square matrix or linear operator $A$ is a $(1 + \sqrt{2})$-spectral
set for $A$; that is, for any function $f$ analytic in the interior of the numerical range $W(A)$
and continuous on its boundary,
the inequality $\| f(A) \| \leq (1 + \sqrt{2} ) \| f \|_{W(A)}$ holds, where the norm
on the left is the operator 2-norm and $\| f \|_{W(A)}$ on the right denotes the supremum of $| f(z) |$
over $z \in W(A)$.  In this paper, we show how the arguments in their paper can be extended
to show that other regions in the complex plane that do {\em not} necessarily contain $W(A)$
are $K$-spectral sets for a value of $K$ that may be close to $1 + \sqrt{2}$. 
We also find some special cases in which the constant $(1 + \sqrt{2})$ for $W(A)$ can be replaced by $2$,
which is the value conjectured by Crouzeix.
\end{abstract}

\section{Introduction}
Let $A$ be a square matrix or a bounded linear operator on a Hilbert space.
In \cite{CP2017}, Crouzeix and Palencia establish that closure of the numerical range of $A$,
\[
W(A) = \{ \langle Aq,q \rangle : \langle q,q \rangle = 1 \} ,
\]
is a $(1 + \sqrt{2})$-spectral set for $A$; that is, for any function $f$ analytic
in the interior of $W(A)$ and continuous on its boundary, 
\[
\| f(A) \| \leq ( 1 + \sqrt{2} ) \| f \|_{W(A)} ,
\]
where the norm on the left is the operator 2-norm and $\| f \|_{W(A)}$ on the right
denotes the supremum of $| f(z) |$ over $z \in W(A)$.  Crouzeix's conjecture is that
$\| f(A) \| \leq 2 \| f \|_{W(A)}$ \cite{Crouzeix2}.

In this paper we show how the arguments in \cite{CP2017} can be extended to show that 
other regions $\Omega$ in the complex plane that do {\em not} necessarily contain $W(A)$ are
$K$-spectral sets for a value of $K$ that may be close to $1 + \sqrt{2}$.
We present numerical results that show what these values $K$ are for various regions
that contain the spectrum of $A$ but not necessarily all of $W(A)$.  In particular we 
consider disks of various sizes containing the spectrum of $A$.

In part of the paper, we limit the discussion to $n$ by $n$ matrices.  In this case, 
for any proper open subset $\Omega$ of $\mathbb{C}$ containing the spectrum of $A$,
there is a function $f$ that attains 
$\sup_{\hat{f} \in {\cal H} ( \Omega )} \| \hat{f} (A) \| / \| \hat{f} \|_{\Omega}$,
where ${\cal H} ( \Omega )$ denotes the set of analytic functions in $\Omega$.  Furthermore,
if $\Omega$ is simply connected, the form of $f$ is known \cite{Crouzeix1,Earl,GladerLindstrom}:
\[
f(z) = B \circ \varphi ( z) ,
\]
where $\varphi$ is any bijective conformal mapping from $\Omega$ to the unit disk ${\cal D}$ and
$B$ is a Blaschke product of degree at most $n-1$,
\begin{equation}
B(z) = e^{i \theta} \prod_{j=1}^{n-1} \frac{z - \alpha_j}{1 - \bar{\alpha}_j z} ,~~~| \alpha_j | \leq 1 .
\label{Blaschke}
\end{equation}
[Note that we have allowed $| \alpha_j | = 1$ in this definition so that the degree
of $B(z)$ can be less than $n-1$, since factors with $| \alpha_j | = 1$ are just unit scalars.]
We show that for this ``optimal'' $B$, if $\| B ( \varphi (A) ) \| > 1$ and if $y$ is a right
singular vector of $B( \varphi (A) )$ corresponding to the largest singular value, then 
$\langle By,y \rangle = 0$.  Using this result,
we are able to replace the bound $1 + \sqrt{2}$ in \cite{CP2017} by $2$ in some special cases.
In particular, we give a new proof of the Okubo and Ando result \cite{OkuboAndo} that if
$W(A) \subset {\cal D}$ then ${\cal D}$ is a 2-spectral set for $A$.  

The paper is organized as follows.  In Section 2, we prove the basic theorems extending 
results in \cite{CP2017} to other regions $\Omega$ containing the spectrum of $A$ but not 
necessarily all of $W(A)$.  The proofs of these theorems involve minor modifications of the
arguments in \cite{CP2017}.  We then limit the discussion to $n$ by $n$ matrices.
In Section 3, we use results from Section 2 to show that if one
replaces $W(A)$ by an appropriate region $\Omega$ a distance about $\epsilon$ inside $W(A)$,
then this region is a $K$-spectral set for $K = 1 + \sqrt{2} + O( \epsilon )$.
Section 4 contains numerical studies of the values of $K$ derived in Section 2,
when $\Omega$ is a disk containing the spectrum of $A$.
In Section 5 we derive a property of the optimal Blaschke product (\ref{Blaschke}). 
This is used in Section 6 to describe cases in which the bound $1 + \sqrt{2}$ can be replaced by $2$
and in which the bound in Theorem \ref{theorem:1} can be replaced by $2 + \delta$ from Lemma \ref{lemma:1}.
Section 7 describes a case in which the bound of Lemma \ref{lemma:1} is sharp, and in Section 8
we mention remaining open problems.

\section{Extensions of the Crouzeix-Palencia Arguments}
Let $\Omega$ be a region with smooth boundary containing the spectrum of $A$ in its interior.
Any function $f \in {\cal A} ( \Omega ) := {\cal H}( \Omega ) \cap C( \mbox{cl}( \Omega ))$, 
is defined via the Cauchy integral formula as
\[
f(A) = \frac{1}{2 \pi i} \int_{\partial \Omega} ( \sigma I - A )^{-1} f( \sigma )\,d \sigma .
\]
If we parameterize $\partial \Omega$ by arc length $s$ running from $0$ to $L$, then this
can be written as
\[
f(A) = \int_0^L \frac{\sigma' (s)}{2 \pi i} R( \sigma (s) , A ) f( \sigma (s))\,ds ,
\]
where $R( \sigma , A )$ is the resolvent, $( \sigma I - A )^{-1}$.

A key idea in \cite{CP2017} was to look at
\begin{equation}
g(A) = \int_0^L \frac{\sigma' (s)}{2 \pi i} R( \sigma (s) , A ) \overline{f( \sigma (s) )}\,ds .
\label{cauchytransform}
\end{equation}
Note that since $\overline{f( \sigma (s) )}$ is not analytic, one cannot apply the Cauchy integral
formula in its simplest form to the integral in (\ref{cauchytransform}).
Crouzeix and Palencia analyzed the operator
\[
S := f(A) + g(A )^{*} = \int_0^L \mu ( \sigma (s) , A ) f( \sigma (s) )\,ds , 
\]
where the Hermitian operator $\mu ( \sigma (s) , A )$ is
\begin{equation}
\mu ( \sigma (s) , A ) = \frac{\sigma' (s)}{2 \pi i} R( \sigma (s), A ) + \left[ 
\frac{\sigma' (s)}{2 \pi i} R( \sigma (s), A ) \right]^{*} . \label{mu}
\end{equation}
They argued that if $\Omega$ is a convex region that contains $W(A)$, then $\mu ( \sigma , A )$ is positive
semidefinite for $\sigma \in \partial \Omega$.  We are interested in regions $\Omega$
that do not necessarily contain $W(A)$, so we will define
\begin{equation}
M( \sigma , A ) := \mu ( \sigma , A ) - \lambda_{\min} ( \mu ( \sigma , A ) ) I , \label{M}
\end{equation}
where $\lambda_{\min} ( \mu ( \sigma , A ))$ is the infimum of the spectrum of 
$\mu ( \sigma , A )$ at the point $\sigma$ on $\partial \Omega$.  Thus, by definition, 
$M( \sigma , A )$ is positive semidefinite on $\partial \Omega$.

Using the same method of proof as in \cite{CP2017}, we establish the following:

\begin{lemma}
Let $\Omega$ be a region with smooth boundary containing the spectrum of $A$ in its interior.  
For $f \in {\cal A} ( \Omega )$ with $\| f \|_{\Omega} = 1$, let
\begin{equation}
S = f(A) + g(A )^{*} + \gamma I ,~~~ \gamma := - \int_0^L \lambda_{\min}( \mu ( \sigma (s), A ) ) 
f( \sigma(s) )\,ds .
\label{S}
\end{equation}
Then $\| S \| \leq 2 + \delta$ , where
\begin{equation}
\delta =  - \int_0^L \lambda_{\min} ( \mu ( \sigma (s),A) )\,ds . \label{delta}
\end{equation}
\label{lemma:1}
\end{lemma}

\begin{proof}
Let $u$ and $v$ be any two unit vectors and, for convenience, write $M(s)$ for $M( \sigma (s) , A )$
and $\lambda_{\min} (s)$ for $\lambda_{\min} ( \mu ( \sigma (s),A ) )$
in (\ref{M}).  Then
\begin{eqnarray*}
| \langle Sv,u \rangle | & = & \left| \int_0^L \langle M(s) v , u \rangle 
f( \sigma (s))\,ds \right| \\
 & \leq & \int_0^L | \langle M(s) v , u \rangle |\,ds~~~
\mbox{(since $\| f \|_{\Omega} = 1$)} \\
 & \leq & \int_0^L \langle M(s) u , u \rangle^{1/2}
  \cdot \langle M(s) v , v \rangle^{1/2}\,ds~~~
\mbox{(Cauchy-Schwarz, since $M(s)$ is positive semidefinite)} \\
 & \leq & \left( \int_0^L \langle M(s) u, u \rangle\,ds \right)^{1/2}
  \left( \int_0^L \langle M(s) v, v \rangle\,ds \right)^{1/2}~~~
\mbox{(Bunyakovskii's inequality)} \\
 & = & \left\langle \left( \int_0^L M(s)\,ds \right) u , u \right\rangle^{1/2}
  \left\langle \left( \int_0^L M(s)\,ds \right) v , v \right\rangle^{1/2} \\
 & = & \left\langle \left( 2 - \int_0^L \lambda_{\min} ( s )\,ds \right) u , 
u \right\rangle^{1/2}
  \left\langle \left( 2 - \int_0^L \lambda_{\min} ( s)\,ds \right) v , 
v \right\rangle^{1/2}~~~
\mbox{(since $\int_0^L \mu ( \sigma (s),A )\,ds = 2I$)} \\
 & = & 2 + \delta .
\end{eqnarray*}
\end{proof}

\noindent
{\em Remark 1:}  Note that $\delta$ in (\ref{delta}) can be positive or negative (but it cannot
be less than $-2$).  It is $0$
if $\Omega = W(A)$, positive if $\Omega$ is a proper subset of $\mbox{int}(W(A))$, and negative 
if $\Omega$ is convex and $\mbox{cl}(W(A))$ is a proper subset of $\Omega$.
This follows from the fact shown in \cite{CP2017} that if $\tau$ lies on the tangent line to $W(A)$ 
at a point $\sigma \in \partial W(A)$, the infimum of the spectrum of the Hermitian part of
$( \sigma' (s) / ( \pi i ) ) R( \tau , A )$
is $0$, while on the side of this line that does not contain $W(A)$ it is positive and on the side
that does contain $W(A)$ it is negative.
\vspace{.1in}

\noindent
{\em Remark 2:}  The region $\Omega$ in Lemma \ref{lemma:1} need not be simply connected.  For example, it
could consist of a union of disks or other smooth regions, each of which encloses a part of the
spectrum.  
\vspace{.1in}

The remainder of \cite{CP2017} is aimed at relating $\| f(A) + g(A )^{*} \|$ to $\| f(A) \|$.
We note that in {\em many} numerical experiments in which $A$ is an $n$ by $n$ matrix
and $f$ is the function with $\| f \|_{W(A)} = 1$ that maximizes $\| \hat{f} (A) \|$
over all $\hat{f} \in {\cal A} (W(A))$ with $\| \hat{f} \|_{W(A)} = 1$ (or, at least,
our best attempt at computing this function via numerical optimization of roots of 
a Blaschke product), 
we have {\em always} found that $\| f(A) \| \leq \| f(A) + g(A )^{*} \|$.  If this could
be proved, then it would establish Crouzeix's conjecture.  Since we do not know a
proof, however, more work is needed to bound $\| f(A) \|$ in terms
of $\| f(A) + g(A )^{*} \|$.

We now assume that $\Omega$ is a bounded {\em convex} domain with smooth boundary.
It is shown in \cite{CP2017} that if $g(z)$ is defined in $\Omega$ by
\begin{equation}
g(z) = \frac{1}{2 \pi i} \int_{\partial \Omega} \frac{\overline{f( \sigma )}}{\sigma - z}\,d \sigma ,
\label{gofz}
\end{equation}
then $g \in {\cal A} ( \Omega )$ (when $g$ is extended continuously to $\partial \Omega$),
$g(A)$ satisfies (\ref{cauchytransform}), and 
\begin{equation}
\| g \|_{\Omega} \leq \| f \|_{\Omega} . \label{glef}
\end{equation}
It is further shown that $g( \partial \Omega ) := \{ g( \sigma ) :
\sigma \in \partial \Omega \}$ is contained in the convex hull of $\overline{f( \partial \Omega )} :=
\{ \overline{f( \sigma )} : \sigma \in \partial \Omega \}$.

For any bounded set $\Omega$ containing the spectrum of $A$ in its interior, with the spectrum
of $A$ bounded away from $\partial \Omega$, there is a minimal
constant $c_{\Omega} (A)$ (which, for convenience, we denote as simply $c_{\Omega}$)
such that for all $f \in {\cal A} ( \Omega )$,
\[
\| f(A) \| \leq c_{\Omega} \| f \|_{\Omega} .
\]
One such constant can be derived from the Cauchy integral formula:
\begin{equation}
\| f(A) \| \leq \frac{1}{2 \pi} \left( \int_{\partial \Omega} \| ( \sigma I - A )^{-1} \|~
| d \sigma | \right) \| f \|_{\Omega} , \label{Cauchybound}
\end{equation}
but this usually is not optimal.  It was shown in \cite{Delyon} that even if the spectrum
of $A$ is not bounded away from $\partial \Omega$, if $\Omega \supset W(A)$, then such a 
constant exists and is finite.
The following theorem uses Lemma \ref{lemma:1} and (\ref{glef})
to obtain a different upper bound on $c_{\Omega}$.

\begin{theorem}\footnote{The authors thank Felix Schwenninger and an anonymous referee
for an improvement to the bound obtained in the original version of this theorem.}
Let $\Omega$ be a convex domain with smooth boundary containing the spectrum of $A$
in its interior.  Then
\begin{equation}
c_{\Omega} \leq 1 + \frac{\delta}{2} + \sqrt{2 + \delta + \delta^2 / 4 + \hat{\gamma}} ,
\label{cOmega}
\end{equation}
where 
\begin{equation}
\delta = - \int_0^L \lambda_{\min} ( \mu ( \sigma (s) , A ))\,ds ,~~~
\hat{\gamma} = \int_0^L | \lambda_{\min} ( \mu ( \sigma (s) , A )) |\,ds .
\label{delta_hat}
\end{equation}
\label{theorem:1}
\end{theorem}

\begin{proof}
Let $f \in {\cal A} ( \Omega )$ satisfy $\| f \|_{\Omega} = 1$.  From (\ref{S}), we can
write
\[
f(A )^{*} = S^{*} - (g(A) + \bar{\gamma} I ) .
\]
Multiply by $f(A )^{*} f(A)$ on the left and by $f(A)$ on the right to obtain
\[
[ f(A )^{*} f(A) ]^2 = f(A )^{*} f(A) S^{*} f(A) - f(A )^{*} f(A) ( g(A) + \bar{\gamma} I) f(A) .
\]
Now take norms on each side and use the fact that the norm of any function of $A$ is
less than or equal to $c_{\Omega}$ times the supremum of that function on $\Omega$ to find:
\begin{eqnarray*}
\| f(A) \|^4 & \leq & c_{\Omega}^3 \| S^{*} \| + c_{\Omega} \| h(A) \| ,~~~
h(z) := f(z)(g(z) + \bar{\gamma} ) f(z) , \\
 & \leq & c_{\Omega}^3 (2 + \delta ) + c_{\Omega}^2 ( 1 + \hat{\gamma} ) ,~~
\mbox{(since $\| S^{*} \| \leq 2 + \delta$ and $\|h \|_{\Omega} \leq 1 + \hat{\gamma}$).}
\end{eqnarray*}
Since this holds for all $f \in {\cal A} ( \Omega )$ with $\| f \|_{\Omega} = 1$, it
follows that
\[
c_{\Omega}^4 \leq c_{\Omega}^3 ( 2 + \delta ) + c_{\Omega}^2 ( 1 + \hat{\gamma} ),
\]
and solving the quadratic inequality $c_{\Omega}^2 - (2 + \delta ) c_{\Omega} -
(1 + \hat{\gamma} ) \leq 0$ for $c_{\Omega}$ gives the desired result (\ref{cOmega}).
\end{proof}

\section{A Perturbation Result}
How does $\lambda_{\min} ( \mu ( \sigma , A ) )$ vary as $\sigma$ moves inside or
outside $\partial W(A)$?  Is a region just slightly inside $W(A)$ a $K$-spectral set
for $A$ for some $K$ that is just slightly greater than $1 + \sqrt{2}$?

We now assume that $A$ is an $n$ by $n$ matrix and that no eigenvalue of $A$ lies
on $\partial W(A)$.
Consider a curve consisting of points $\tilde{\sigma} ( \tilde{s} )$,
where $\tilde{s}$ runs from $0$ to $\tilde{L} < L$ and $\tilde{\sigma}' ( \tilde{s} ) = \sigma' (s)$,
for $\sigma (s)$ on $\partial W(A)$ close to $\tilde{\sigma}( \tilde{s} )$.
For example, we might take a ``center'' point of $W(A)$ as the origin and for $\theta \in [0 , 2 \pi )$ write
$\sigma ( \theta ) = r( \theta ) e^{i \theta}$, $\tilde{\sigma} ( \theta ) = 
(1 - \epsilon ) r( \theta ) e^{i \theta }$, where $r( \theta )$ is the distance from the center
point to the point on $\partial W(A)$ with argument $\theta$, and $\epsilon > 0$ is small.  
Then $\left. \frac{d \tilde{\sigma}}{d \tilde{s}} \right|_{\tilde{s}} = 
\left. \frac{d \sigma}{ds} \right|_{s}$ if $\tilde{s} = (1 - \epsilon ) s$.

From the Wielandt-Hoffman theorem \cite{HW1953}, 
\begin{equation}
| \lambda_{\min} ( \mu ( \sigma , A )) - \lambda_{\min} ( \mu ( \tilde{\sigma} , A ) ) |
\leq \| \mu ( \sigma , A ) - \mu ( \tilde{\sigma} , A ) \| , \label{WielandtHoffman}
\end{equation}
and the same inequality holds for the difference between every pair of ordered eigenvalues
of $\mu ( \sigma , A )$ and $\mu ( \tilde{\sigma} , A )$.
The matrix on the right is the Hermitian part of
\[
\frac{\sigma' (s)}{\pi i} \left[ R( \sigma , A ) - R( \tilde{\sigma} , A ) \right] ,
\]
and from the first resolvent identity this is equal to the Hermitian part of
\[
( \tilde{\sigma} - \sigma ) \left[ \frac{\sigma' (s)}{\pi i} R( \sigma , A ) R( \tilde{\sigma} , A )
\right] .
\]
Therefore since, from \cite{CP2017}, $\lambda_{\min} ( \mu ( \sigma , A ) ) = 0$, it follows from 
(\ref{WielandtHoffman}) that
\[
\lambda_{\min} ( \mu ( \tilde{\sigma} , A )) \geq 
- | \tilde{\sigma} - \sigma | \cdot
\left\| \left[ \frac{\sigma' (s)}{2 \pi i} R( \sigma , A ) R( \tilde{\sigma} , A )
\right] + \left[ \frac{\sigma' (s)}{2 \pi i} R( \sigma , A ) R( \tilde{\sigma} , A ) \right]^{*}
\right \| .
\]
As $\epsilon \rightarrow 0$, this shows that the smallest eigenvalue of $\mu ( \tilde{\sigma} , A )$
is greater than or equal to $-C \epsilon$ for a positive constant $C$, and $\delta$ and 
$\hat{\gamma}$ in (\ref{delta_hat}) are therefore $O( \epsilon )$, so that $c_{\Omega} \leq 1 + \sqrt{2} + O( \epsilon )$.

\section{Numerical Studies}
To further study the behavior of $\lambda_{\min} ( \mu ( \sigma , A ) )$ for $\sigma$
inside or outside $W(A)$, one can make a plot of $\lambda_{\min}$ vs. arc length $s$ or angle 
$\theta (s)$, for $\sigma (s)$ on various curves.  [Note that the value of $\lambda_{\min}$
depends not only on the location of $\sigma$ but also on the curve on which $\sigma$ is 
considered to lie.]  Instead of taking curves on
which $\sigma'$ matches the derivative at some associated point on $\partial W(A)$, we will now
take $\sigma$ to lie on a circle.   We first determine the center $c$ and radius $r$
of the smallest circle enclosing the spectrum of $A$ and then determine values of $\lambda_{\min} 
( \mu ( \sigma , A ) )$ at points $\sigma$ on circles about $c$ of radius $R > r$.  On such circles,
$\sigma = c + R e^{i \theta}$, $0 \leq \theta < 2 \pi$, and $s = R \theta$ so that
\[
\sigma (s) = c + R e^{i s/R} ,~~~\sigma' (s) = i e^{i s/R} .
\]

Figure \ref{fig:1} shows the eigenvalues (x's) and numerical range (solid curve) of
a random complex upper triangular matrix $A$ of dimension $n=12$.  It also shows the 
circles on which we computed $\lambda_{\min} ( \mu ( \sigma , A ))$ (dashed circles).
Figure \ref{fig:2} shows the values of $\lambda_{\min} ( \mu ( \sigma (s) , A ))$ plotted
vs. arc length $s$ on each of these circles, where the bottom curve corresponds 
to the innermost circle and the curves move up as the circles become larger.  From the figure, 
it can be seen that $\lambda_{\min}$ decreases rapidly as $\sigma$
moves inside $W(A)$ towards the spectrum, but it grows very slowly as $\sigma$ moves outside
$W(A)$.

\begin{figure}[ht]
\centerline{\epsfig{file=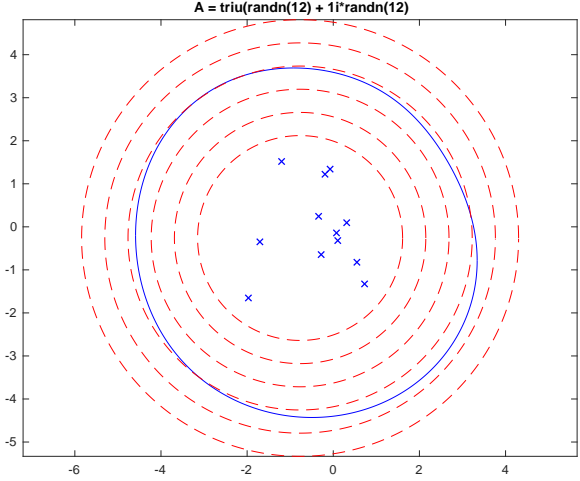,width=3in}}
\caption{Eigenvalues (x's) and numerical range (solid curve) of a random complex
upper triangular matrix $A$ of dimension $n=12$.  The dashed circles are the ones
on which $\lambda_{\min} ( \mu ( \sigma , A ))$ was computed.} \label{fig:1}
\end{figure}

\begin{figure}[ht]
\centerline{\epsfig{file=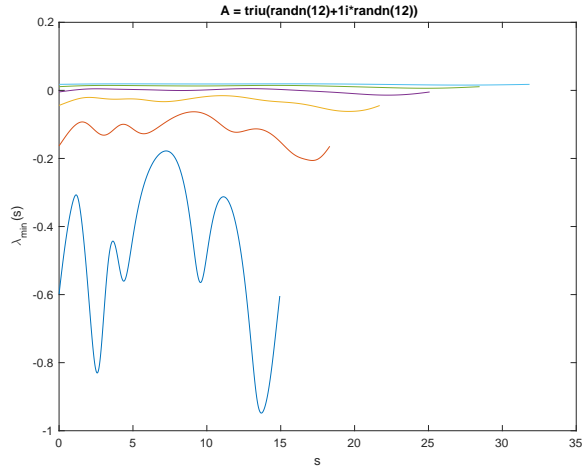,width=3in}}
\caption{Plot of $\lambda_{\min} ( \mu ( \sigma (s) , A ))$ vs. arc length $s$ on each
of the dashed circles in Figure \ref{fig:1}.  The bottom curve corresponds to the 
innermost circle and the curves move up as the circles become larger.} \label{fig:2}
\end{figure}

Figures \ref{fig:3} and \ref{fig:4} show the same results for a $3$ by $3$ perturbed
Jordan block:
\begin{equation}
A = \left[ \begin{array}{ccc} 0 & 1 & 0 \\ 0 & 0 & 1 \\ 0.1 & 0 & 0 \end{array} \right] .
\label{Jordan}
\end{equation} 
Again one can see that $\lambda_{\min} ( \mu ( \sigma , A ))$
decreases rapidly as $\sigma$ moves toward the spectrum but grows slowly as it moves outside $W(A)$.

\begin{figure}[ht]
\centerline{\epsfig{file=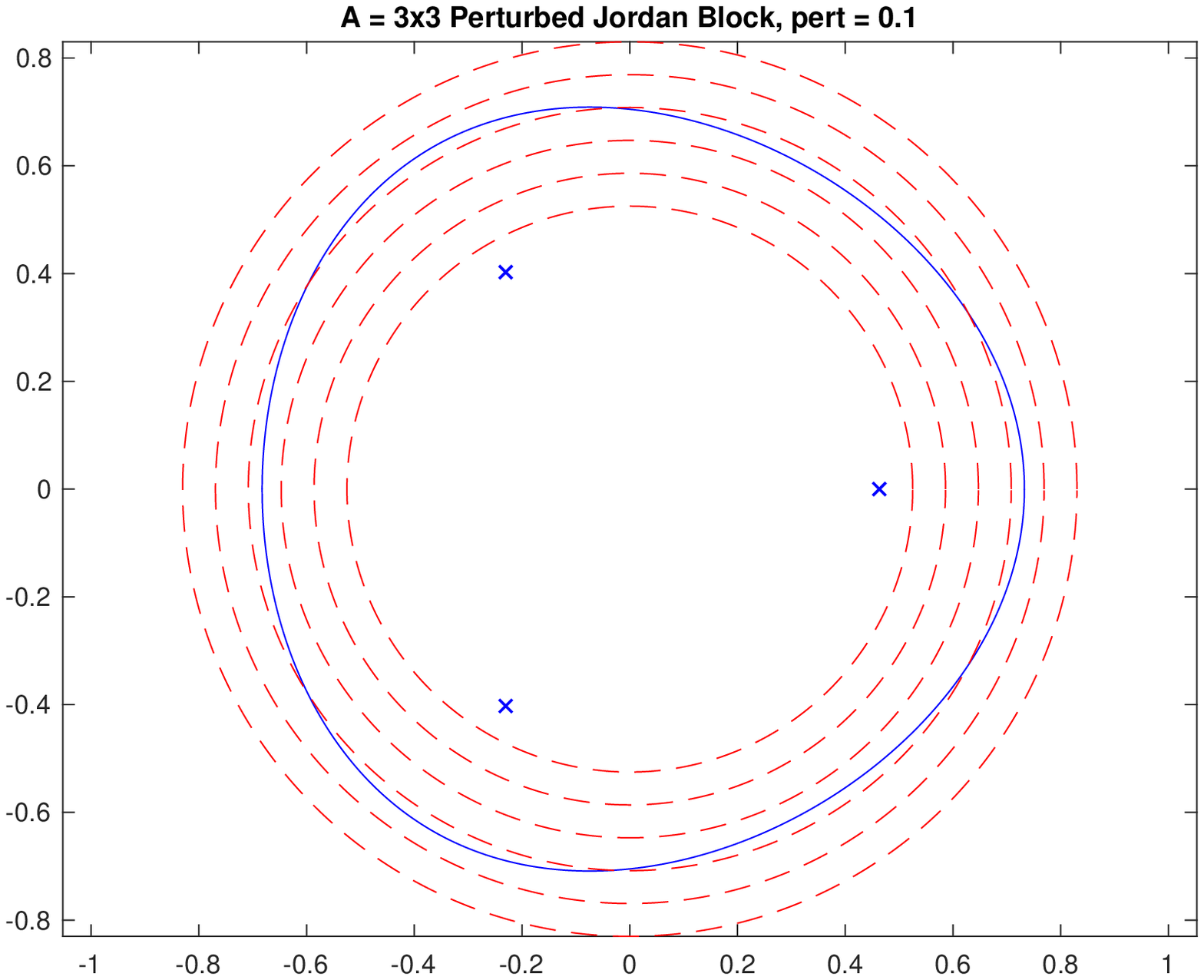,width=3in}}
\caption{Eigenvalues (x's) and numerical range (solid curve) of the $3$ by $3$ perturbed
Jordan block (\ref{Jordan}).  The dashed circles are the ones
on which $\lambda_{\min} ( \mu ( \sigma , A ))$ was computed.} \label{fig:3}
\end{figure}

\begin{figure}[ht]
\centerline{\epsfig{file=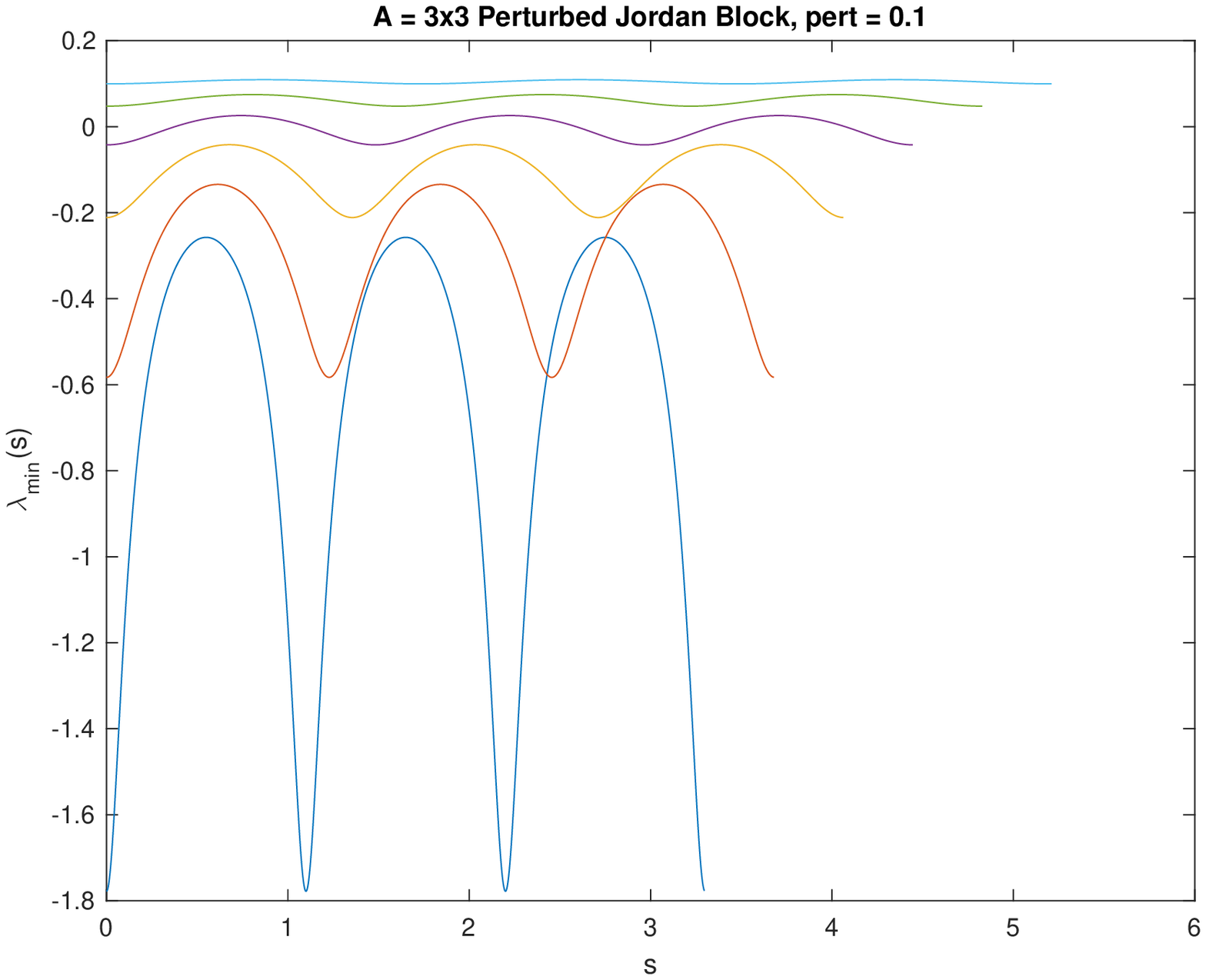,width=3in}}
\caption{Plot of $\lambda_{\min} ( \mu ( \sigma (s) , A ))$ vs. arc length $s$ on each
of the dashed circles in Figure \ref{fig:3}.  The bottom curve corresponds to the 
innermost circle and the curves move up as the circles become larger.} \label{fig:4}
\end{figure}

Table \ref{table:1} shows the values of $\delta$ and $\hat{\gamma}$ in (\ref{delta_hat})
and the upper bound on
$c_{\Omega}$ in (\ref{cOmega}) (labeled $K_{\delta}$) for each
of the disks (starting with the smallest) in both problems.  For comparison,
we also include the upper bound on $c_{\Omega}$ based on the Cauchy integral
formula and the resolvent norm in (\ref{Cauchybound}) (labeled $K_{\mbox{Cauchy}}$).
In all cases, $K_{\delta} < K_{\mbox{Cauchy}}$.

\begin{table}[ht]
\begin{center}
\begin{tabular}{|r|r|r|r|r|r|r|r|} \hline
\multicolumn{4}{|c|}{Random} & \multicolumn{4}{c|}{Perturbed} \\ 
\multicolumn{4}{|c|}{Upper Triangular} & \multicolumn{4}{c|}{Jordan Block} \\ \hline
\multicolumn{1}{|c|}{$\delta$} & \multicolumn{1}{|c|}{$\hat{\gamma}$} & \multicolumn{1}{c|}{$K_{\delta}$} &
\multicolumn{1}{c|}{$K_{\mbox{Cauchy}}$} &
\multicolumn{1}{|c|}{$\delta$} & \multicolumn{1}{|c|}{$\hat{\gamma}$} & \multicolumn{1}{c|}{$K_{\delta}$} &
\multicolumn{1}{c|}{$K_{\mbox{Cauchy}}$}  \\ \hline \hline
7.0485 & 7.0485 & 9.865 & 12.160 & 2.2234 & 2.2234 & 4.884 & 5.635 \\ \hline
2.2291 & 2.2291 & 4.890 & 6.422 & 1.0969  & 1.0969 & 3.669 & 4.416 \\ \hline
0.7209 & 0.7209 & 3.251 & 4.445 & 0.4557  & 0.4557 & 2.950 & 3.741 \\ \hline
0.0377 & 0.1179 & 2.488 & 3.479 & 0.0201  & 0.0946 & 2.465 & 3.291 \\ \hline
-0.3388 & 0.3388     & 2.255 & 2.918 & -0.2999 & 0.2999      & 2.273 & 2.965 \\ \hline
-0.5767 & 0.5767     & 2.155 & 2.555 & -0.5456 & 0.5456       & 2.168 & 2.716 \\ \hline
\end{tabular}
\end{center}
\caption{Values of $\delta$ and $\hat{\gamma}$ in (\ref{delta_hat})
and upper bound ($K_{\delta}$) on $c_{\Omega}$ in 
(\ref{cOmega}) and upper bound ($K_{Cauchy}$) on $c_{\Omega}$ in (\ref{Cauchybound})
for disks in Figures \ref{fig:1} and \ref{fig:3}.} \label{table:1}
\end{table}

\section{Optimal Blaschke Products}
From here on we always assume that $A$ is an $n$ by $n$ matrix.  Then
if $\Omega$ is any simply connected proper open subset of $\mathbb{C}$ containing the 
spectrum of $A$, then there is a function $f$ such that $\| f \|_{\Omega} = 1$
and $\| f (A) \| = c_{\Omega}$.  This function $f$ is of the form $B \circ \varphi$,
where $\varphi$ is any bijective conformal mapping from $\Omega$ to the unit disk
${\cal D}$, and $B$ is a Blaschke product of the form (\ref{Blaschke}).

While we do not know an analytic formula for the roots $\alpha_j$, $j=1, \ldots , n-1$,
of this optimal Blaschke product, the following theorem describes one property of the
optimal $B$:

\begin{theorem}
Let $\Psi$ be an $n$ by $n$ matrix whose spectrum is inside the unit disk ${\cal D}$ 
and let $B$ be a Blaschke product of degree at most $n-1$ that maximizes 
$\| \hat{B} ( \Psi ) \|$ over all Blaschke products $\hat{B}$ of degree at most $n-1$.  Assume
that $\| B( \Psi ) \| > 1$.  Then, if $v_1$ is a right singular vector of $B( \Psi )$
corresponding to the largest singular value $\sigma_1$, then
\begin{equation}
\langle B( \Psi ) v_1 , v_1 \rangle = 0 . \label{orthogonal}
\end{equation}
\label{theorem:2}
\end{theorem}

\begin{proof}\footnote{The authors thank M.~Crouzeix for the nice proof of this theorem,
after we had observed the result in numerical experiments.}
Let $M = B( \Psi )$ where $B$ is the Blaschke product of the form (\ref{Blaschke}) for
which $\| B ( \Psi ) \|$ is maximal.  Then no matrix of the form
\[
( M - \alpha I ) ( I - \bar{\alpha} M )^{-1} ,~~~| \alpha | < 1 ,
\]
can have larger norm than $M$ since this is also a Blaschke product in $\Psi$.
Let $v_1$ be a unit right singular vector of $M$ corresponding to the largest singular
value $\sigma_1$, and define $w = (I - \bar{\alpha} M ) v_1$.  Then
\[
\| (M - \alpha I ) v_1 \| = \| (M - \alpha I ) ( I - \bar{\alpha} M )^{-1} w \| \leq
\sigma_1 \| w \| = \sigma_1 \| (I - \bar{\alpha} M ) v_1 \| .
\]
Squaring both sides, this becomes
\[
\langle M v_1 , M v_1 \rangle - 2 \mbox{Re}( \bar{\alpha} \langle M v_1 , v_1 \rangle ) +
| \alpha |^2 \leq \sigma_1^2 [ 1 - 2 \mbox{Re}( \bar{\alpha} \langle M v_1 , v_1 \rangle ) +
| \alpha |^2 \langle M v_1 , M v_1 \rangle ],
\]
and since $\langle M v_1 , M v_1 \rangle = \sigma_1^2$,
\[
2 ( \sigma_1^2 - 1) \mbox{Re} ( \bar{\alpha} \langle M v_1 , v_1 \rangle ) \leq | \alpha |^2
( \sigma_1^4 - 1 ) .
\]
With the assumption that $\sigma_1 > 1$, dividing by $\sigma_1^2 - 1$ this becomes
\[
2 \mbox{Re} ( \bar{\alpha} \langle M v_1 , v_1 \rangle ) \leq | \alpha |^2 ( \sigma_1^2 + 1 ) .
\]
Choosing $\alpha$ so that $\bar{\alpha} \langle M v_1 , v_1 \rangle = | \alpha |~
| \langle M v_1 , v_1 \rangle |$, we have
\[
2 | \alpha |~ | \langle M v_1 , v_1 \rangle | \leq | \alpha |^2 ( \sigma_1^2 + 1 ) ,
\]
and letting $| \alpha | \rightarrow 0$, this implies that $| \langle M v_1 , v_1 \rangle | = 0$.
\end{proof}
\section{Some Cases in Which $1 + \sqrt{2}$ can be Replaced by $2$}

It was noted after Lemma \ref{lemma:1} that in all numerical experiments that we have performed --
determining $f = B \circ \varphi$ by first finding a conformal mapping $\varphi$
from $W(A)$ to ${\cal D}$ and then using an optimization code with many different 
initial guesses to try to find the roots $\alpha_1 , \ldots , \alpha_{n-1}$ in (\ref{Blaschke}) 
that maximize $\| f(A) \| = \| B( \varphi (A) ) \|$, and then finding the matrix $g(A)$ in
(\ref{cauchytransform}) that corresponds to this $f$ -- it has always been the case
that $\| f(A) \| \leq \| f(A) + g(A )^{*} \|$.  We now describe some cases in which 
that can be proved and also cases in which it can be shown that for other regions 
$\Omega$, $\| f(A) \| \leq \| S \|$ in Lemma \ref{lemma:1}.

\subsection{$\Omega$ a Disk}

If $\Omega$ is a closed disk with center $c$, the formula (\ref{gofz}) for $g(z)$ can be 
evaluated very simply:  For $z$ inside $\Omega$, $g(z) \equiv \overline{f(c)}$ 
\cite[p.~205; e.g.]{Remmert}.
Thus $g(A) = \overline{f(c)} I$.  If $f = B \circ \varphi$ is a 
function that maximizes $\| f(A) \|$ over all functions with $\| f \|_{\Omega} = 1$,
if $\| f(A) \| > 1$,
and if $v_1$ is a unit right singular vector of $f(A)$ corresponding to the largest 
singular value and $u_1 = f(A) v_1 / \| f(A) v_1 \|$ 
is the corresponding unit left singular vector, then
\[
u_1^* ( f(A) + g(A )^{*} + \gamma I ) v_1 = u_1^* f(A) v_1 = \| f(A) \| ,
\]
since, by Theorem \ref{theorem:2}, $u_1^{*} v_1 = 0$.  It follows that $\| f(A) \|$
is less than or equal to $\max \{ 1,~\| S \| \}$ in (\ref{S}).  
This provides a new proof of the statement:  
\begin{equation}
{\mbox
{\em If $W(A)$ is contained in a closed disk $\Omega$, then $\Omega$ is a 2-spectral set for $A$},
} \label{statement}
\end{equation}
since in this case $\delta$ in (\ref{delta}) is less than or equal to $0$.  [A priori,
the assumption that the spectrum of $A$ is contained in the interior of $W(A)$ is
needed, but this can be easily avoided.]  Note also that $c_{\Omega} < 2$ if $W(A)
\neq \Omega$ since in this case $\delta$ is negative.
The proof of (\ref{statement}) in \cite{Crouzeix1} was based on a paper of 
Ando \cite{Ando}, and a similar construction had been carried out by Okubo and 
Ando \cite{OkuboAndo} in a more general setting.  Our proof is simpler, but the
previous proofs showed, in addition, that $\Omega$ is a {\em complete} 2-spectral set for $A$.

Even if $\Omega$ does not contain all of $W(A)$, the estimate 
$\| f(A) \| \leq \max \{ 1,~2 + \delta \}$ holds.
This provides a better upper bound
on $c_{\Omega}$ in the experiments of Section 4.  For comparison, Table \ref{table:2}
lists this upper bound on $c_{\Omega}$ and also the largest value returned by our
optimization code, which we believe to be the true value of $c_{\Omega}$ but, at least, it
is a lower bound.  In some cases, these are quite close.

\begin{table}[ht]
\begin{center}
\begin{tabular}{|r|r|r|r|} \hline
\multicolumn{2}{|c|}{Random} & \multicolumn{2}{c|}{Perturbed} \\ 
\multicolumn{2}{|c|}{Upper Triangular} & \multicolumn{2}{c|}{Jordan Block} \\ \hline
\multicolumn{1}{|c|}{$2+ \delta$} & \multicolumn{1}{c|}{$\| B( \varphi (A) ) \|$} &
\multicolumn{1}{|c|}{$2+ \delta$} & \multicolumn{1}{c|}{$\| B( \varphi (A) ) \|$} 
\\ \hline \hline
9.049 & 4.400  & 4.224 & 3.625  \\ \hline
4.230 & 2.584  & 3.097 & 2.910  \\ \hline
2.721 & 2.058  & 2.456 & 2.387  \\ \hline
2.038 & 1.752  & 2.021 & 1.993  \\ \hline
1.662 & 1.538  & 1.701 & 1.690  \\ \hline
1.424 & 1.372  & 1.455 & 1.451  \\ \hline
\end{tabular}
\end{center}
\caption{Upper bound $2 + \delta$ on $c_{\Omega}$ and lower bound $\| B( \varphi (A)) \|$
found by numerical optimization of $B$ for disks in Figures \ref{fig:1} and \ref{fig:3}.} \label{table:2}
\end{table}

\subsubsection{A Different Bound on $c_{W(A)}$}
Since we know that a disk containing the spectrum of a matrix $\Psi$ is a 
$\max \{ 1, 2 + \delta \}$-spectral set for $\Psi$, we can use this to obtain (numerically)
a different bound on $c_{W(A)}$.  Let $\varphi$ be a bijective conformal mapping
from $W(A)$ to the unit disk ${\cal D}$.  Then ${\cal D}$ is a $K$-spectral set
for $\varphi (A)$, where $K = \max \{ 1, 2 + \delta_{\varphi (A)} \}$, and 
\[
\delta_{\varphi (A)} = - \int_0^{2 \pi} \lambda_{\min} ( \mu ( \sigma (s) , \varphi (A) ) )\,ds , 
\]
where $\sigma (s) = e^{is}$.  It follows that $W(A)$ is a $K$-spectral set for $A$,
with the same value of $K$, since for any $f \in {\cal A} (W(A))$,
\[
\| f(A) \| = \| f \circ \varphi^{-1} ( \varphi (A) ) \| \leq 
K \| f \circ \varphi^{-1} \|_{\cal D} = K \| f \|_{W(A)} .
\]

Using this result, one can determine numerically better bounds on $c_{W(A)}$ for the problems
considered in Section 4.  For the $12$ by $12$ random upper triangular matrix,
$\delta_{\varphi (A)} = 0.0113$, so $W(A)$ is a $2.0113$ spectral set for $A$.
For the $3$ by $3$ perturbed Jordan block, $\delta_{\varphi (A)} = -0.0013$, so
$W(A)$ is a $1.9987$-spectral set for $A$.  It is an open question whether such
bounds can be determined theoretically, and the numerical result for the first problem
suggests that this will not be a way to prove Crouzeix's conjecture.
  
\subsection{Matrices for which Crouzeix's Conjecture has been Proved}
Besides matrices whose numerical range is a circular disk, Crouzeix's conjecture has been 
proved to hold for a number of other classes of
matrices -- e.g., $2$ by $2$ matrices \cite{Crouzeix1}, matrices of the form 
$a I + DP$ or $a I + PD$, where $a \in \mathbb{C}$,
$D$ is a diagonal matrix, and $P$ is a permutation matrix \cite{Choi,GKL}; $3$ by $3$
tridiagonal matrices with elliptic numerical range centered at an eigenvalue
\cite{GKL}; etc.  In all of these cases, while the conformal mapping $\varphi$
may be a complicated function, $\varphi (A)$ is just a linear function of $A$:
$\varphi (A) = \alpha A + \beta I$.  From our experiments, it appears that in
all of these cases the function $g$ corresponding to the optimal $f$ has the form
\[
g(A )^{*} = c_0 I + c_1 ( f(A )^{*} )^{-1} .
\]
Then (assuming $\| f(A) \| > 1$),
\[
| u_1^{*} ( f(A) + g(A )^{*} + \gamma I ) v_1 | = \left| \| f(A) \| + \frac{c_1}{\| f(A) \|} 
\right| .
\]
If $\mbox{Re}( c_1 ) \geq | c_1 |^2 / ( 2 \| f(A) \|^2 )$, then this is greater than or 
equal to $\| f(A) \|$, hence $\| f(A) \| \leq \| S \|$.

\section{A Case in Which the Bound is Sharp}
It was shown in the previous section that if $\Omega$ is a disk containing the
spectrum of $A$ in its interior, then $\max_{f \in {\cal A} ( \Omega )} 
\| f(A) \| / \| f \|_{\Omega} \leq \max \{ 1,~2 + \delta \}$, where $\delta$ is defined in (\ref{delta}).
It turns out that if $A$ is a $3$ by $3$ Jordan block, then this bound is sharp for
all disks of radius less than or equal to $1$ centered at the eigenvalue of $A$, 
which, for convenience, we will take to be $0$.

\begin{theorem}
Let $A$ be a $3$ by $3$ Jordan block with eigenvalue $0$ and let $\Omega$ be any
disk about the origin with radius $r \leq 1$.  Then
\begin{equation}
\max_{f \in {\cal A}( \Omega )} \frac{\| f(A) \|}{\| f \|_{\Omega}} = 2 + \delta ,
\label{equality}
\end{equation}
where $\delta$ is defined in (\ref{delta}).
\end{theorem}

\begin{proof}
The function $f$ that achieves the maximum in (\ref{equality}) is of the form
$B \circ \varphi$, where $B$ is a Blaschke product and $\varphi (z) = z/r$ maps 
$\Omega$ to the unit disk.  Since $\varphi (A)$ is a scalar multiple ($1/r$)
of a Jordan block, it is easy to see that the optimal $B$ is $B(z) = z^2$, assuming
$r \leq 1$.  (For $r=1$, one could take $B(z) = z$ or $B(z) = z^2$ since the norm
of a $3$ by $3$ Jordan block with eigenvalue $0$ and its square are both equal to $1$.)
Thus, the left-hand side of (\ref{equality}) is $1/ r^2$.

Next we show that the right-hand side of (\ref{equality}) is equal to $1/ r^2$.
Since $\sigma (s) = r e^{i s/r}$ on $\partial \Omega$, we can write
\[
\frac{\sigma' (s)}{2 \pi i} R ( \sigma (s), A ) = \frac{e^{i s/r}}{2 \pi} 
( r e^{is/r} I - A )^{-1} = \frac{1}{2 \pi r}
\left[ \begin{array}{ccc} 1 & \frac{e^{-is/r}}{r} & \frac{e^{-2is/r}}{r^2} \\
0 & 1 & \frac{e^{-is/r}}{r} \\ 0 & 0 & 1 \end{array} \right] ,
\]
and
\[
\mu ( \sigma (s) , A ) = \frac{1}{2 \pi r} \left[ \begin{array}{ccc}
2 & \frac{e^{-is/r}}{r} & \frac{e^{-2is/r}}{r^2} \\
\frac{e^{is/r}}{r} & 2 & \frac{e^{-is/r}}{r} \\
\frac{e^{2is/r}}{r^2} & \frac{e^{is/r}}{r} & 2 \end{array} \right] .
\]
It is easy to check that for $r \leq 1$, the smallest eigenvalue of this matrix is
\[
\lambda_{\min} ( \mu ( \sigma (s) , A ) ) = \frac{2 r^2 -1}{2 \pi r^3} ,
\]
independent of $s$.  Since the length of $\partial \Omega$ is $2 \pi r$, 
$\delta = -(2 - 1/ r^2 )$ and $2 + \delta = 1/ r^2$.
\end{proof}

\section{Summary and Open Problems}
We have shown how the arguments in \cite{CP2017} can be modified to give information about
regions that contain the spectrum of $A$ but not necessarily all of $W(A)$.  Perhaps the
most interesting regions are disks about the spectrum, which we have shown to be $K$-spectral
sets for $K = \max \{ 1,~2 + \delta \}$, thereby providing a new proof
that if $W(A) \subset {\cal D}$, then ${\cal D}$ is a 2-spectral set for $A$ \cite{OkuboAndo}.

We derived one property of optimal Blaschke products; i.e., Blaschke products that 
maximize $\| B( \Psi ) \|$ where $\Psi$ is a given matrix whose spectrum is in ${\cal D}$.
Specifically, we showed that if $\| B( \Psi ) \| > 1$, then the left and right singular
vectors of $B( \Psi )$ corresponding to the largest singular value must be orthogonal
to each other.  An interesting open problem is to determine other properties of the optimal $B$.

Perhaps the most interesting question raised is whether it is true that if $f = B \circ \varphi$
is the optimal $f$ then $\| f(A) \| \leq \| f(A) + g(A )^{*} \|$.  A proof of this would
establish Crouzeix's conjecture, and a counterexample might lead to a new line of attack.  

\vspace{.2in}
\noindent
{\bf Acknowledgments:}  The authors thank Michel Crouzeix, Michael Overton, and Abbas Salemi
for many helpful discussions and suggestions.  We especially thank Michel Crouzeix for providing
the proof of Theorem \ref{theorem:2} and Abbas Salemi for pointing out that $\delta$ in
Lemma \ref{lemma:1} could be allowed to take on negative values.  We also thank two
anonymous referees for very helpful comments, especially for an improvement of the 
result in Theorem \ref{theorem:1}.

\end{document}